\documentclass[12p]{amsart}
\usepackage{amssymb}
\usepackage{amsfonts}
\usepackage{amsmath}
\usepackage{amsthm}

\DeclareMathOperator{\diam}{diam}

\newtheorem{theorem}{Theorem}

\newtheorem{proposition}[theorem]{Proposition}

\theoremstyle{definition}
\newtheorem{definition}[theorem]{Definition}
\newtheorem{example}[theorem]{Example}

\newtheorem{question}[theorem]{Question}
\newtheorem*{acknowledgement}{Acknowledgements}

\theoremstyle{remark}
\newtheorem{property}{Property}[theorem]
\newtheorem{case}{Case}
\newtheorem{claim}{Subcase}[case]

\begin{document}

\title{On local fixed or periodic point properties}
\author{Alejandro Illanes}
\address{Instituto de Matem\'{a}ticas, Universidad Nacional Aut\'{o}noma
de M\'{e}xico, Circuito Exterior, Ciudad Universitaria, M\'{e}xico, 04510,
D.F., M\'{e}xico}
\email{illanes@matem.unam.mx}
\author{Pawe\l\ Krupski}
\address{Institute of Mathematics, University of Wroc\l aw, pl.
Grunwaldzki 2/4, 50--384 Wroc\l aw, Poland}
\email{Pawel.Krupski@math.uni.wroc.pl}
\date{\today}

\begin{abstract}
A space $X$ has the local fixed point property $LFPP$, (local
periodic point property $LPPP$) if it has an open basis $\mathcal{B}$
such that, for each $B\in \mathcal{B}$, the closure
$\overline{B}$ has the fixed (periodic) point property.
Weaker versions $wLFPP$, $wLPPP$  are also considered
and examples of metric continua that distinguish all these properties are constructed. We show that for planar or one-dimensional
locally connected metric continua the properties are equivalent.
\end{abstract}

\maketitle

A topological space $X$ has the \textit{fixed point property} ($X\in FPP$)
if each continuous map $f:X\rightarrow X$ has a fixed point; $X$ has the
\textit{periodic point property }($X\in PPP$) if each continuous map $%
f:X\rightarrow X$ has a periodic point.

We can localize these properties in two ways as follows.

\begin{definition} A topological space $X$ has the \textit{local
fixed point property }($X\in LFPP$) if $X$ has an open basis $\mathcal{B}$
(a basis for $LFPP$) such that  $\overline{B}\in FPP$ for each $B\in \mathcal{B}$.

$X$ has the \textit{weak local
fixed point property }($X\in wLFPP$) if $X$ has an open basis $\mathcal{B}$
(a basis for $wLFPP$) such that, for each $B\in \mathcal{B}$ and each
continuous map $f:X\rightarrow X$, whenever $f(\overline{B})\subset B$, then $f$
has a fixed point in $B$~\cite{KOU}.

$X$ has the \textit{local periodic point property }($X\in LPPP$) if $X$ has
an open basis $\mathcal{B}$ (a basis for $LPPP$) such that  $\overline{B}\in PPP$ for each $B\in \mathcal{B}$.

$X$ has the \textit{weak local
periodic point property }($X\in wLPPP$) if $X$ has an open basis $\mathcal{B}$
(a basis for $wLPPP$) such that, for each $B\in \mathcal{B}$ and each
continuous map $f:X\rightarrow X$, whenever $f(\overline{B})\subset B$, then $f$ has a periodic point in $B$~\cite{KOU}.
\end{definition}

 Obviously, $LFPP \Rightarrow wLFPP$ and $LPPP \Rightarrow wLPPP$. Examples~\ref{ex2} and~\ref{ex3} show that the converse implications do not hold.

\bigskip

The main motivation for studying  local fixed or periodic point properties is that they play an essential role in the dynamics of continuous
maps: if a perfect ANR-compactum $X$ has the $wLPPP$, then the set of periodic points of a generic map $f:X\rightarrow X$ has no isolated points and is dense in the set of  chain recurrent
points of $f$~\cite{KOU}.

All spaces that are locally AR's (i.~e., they have a basis $\mathcal B$ such that $\overline{B}\in AR$ for each $B\in\mathcal B$) have the $LFPP$. In particular, all    polyhedra and all Euclidean manifolds have the $LFPP$. A
harmonic fan has the $LPPP$ but it does not have the $wLFPP$~\cite{KOU}.
It is natural to ask for what continua properties $LFPP$, $LPPP$ and their weak versions differ. In this note we
show that for all planar or one-dimensional locally connected metric
continua the  properties are equivalent.

The proof for planar continua is based on two classic theorems.

\begin{proposition}\cite[Theorem (13.1), p. 132]{Bo}\label{P1}
If a locally connected continuum $X$ in the plane does not separate the plane, then $X$ is an absolute retract. In particular, $X\in FPP$.
\end{proposition}

\begin{proposition}\cite[Theorem 5, p. 513]{Ku}\label{P2}
If a locally connected continuum $X$ in the plane separates the plane between two points, then $X$ contains a simple closed curve $S$ which separates the plane between these points. In particular, $S$ is a retract of $X$.
\end{proposition}

\begin{theorem}\label{T1}
Let $X$ be a locally connected metric continuum in the plane. Then
\begin{enumerate}
\item
 $X\in FPP$ if and only if $X\in PPP$,
 \item
 properties $wLFPP$, $LFPP$, $wLPPP$ and $LPPP$ for $X$ are mutually equivalent.
 \end{enumerate}
\end{theorem}

\begin{proof} In order to show the nontrivial implications assume first
that $X\in PPP$. In view of Proposition~\ref{P1}, it suffices to observe that $X$
does not separate the plane $\mathbb{R}^{2}$. Indeed, if $X$ separates the
plane then, by Proposition~\ref{P2},  there exists a
retraction $r:S\rightarrow S$ onto a simple closed curve $S$. Then the
composition $\theta r$, where $\theta $ is a "rotation" of $S$ without
periodic points, would be a self-map of $X$ without periodic points,
contrary to the assumption.

Assume now that $X\in wLPPP$ and let $\mathcal{B}$ be a basis for $wLPPP$. We
are going to define a neighborhood basis $\mathcal{V}_{p}$ at each point $%
p\in X$ such that $\bigcup\limits_{p\in X}\mathcal{V}_{p}$ is a basis for $%
LFPP$.

For any $\varepsilon >0$ there are $B\in \mathcal{B}$ and an open connected
subset $V\subset B$ such that $p\in V\subset \overline{V}\subset B$ and $\diam(B)<
\frac{\varepsilon }{3}$. It is also well known that $V$ can be taken to be
uniformly locally connected~\cite[Theorem 3.3, p. 77]{W}, so the continuum $%
A=\overline{V}$ is locally connected~\cite[Theorem 3.6, p. 79]{W}.

If $A$ does not separate $\mathbb{R}^{2}$, then $A\in FPP$ by Proposition~\ref{P1}
and then we include $V$ in $\mathcal{V}_{p}$.

Suppose $A$ separates $\mathbb{R}^{2}$. We claim that
\begin{equation}\label{eq1} \text{\emph{Each bounded component of $\mathbb R^2\setminus A$ is
contained in $X$.}}
\end{equation}

In fact, if a bounded component $D$ of $\mathbb R^2\setminus A$
satisfies $D\setminus X\neq\emptyset$, then, by Proposition~\ref{P2},
there is a simple closed curve $S\subset A$ that separates $\mathbb{R}^{2}$
between a point $a\in D\setminus X$ and a point $b\in \mathbb R^2\setminus X$. By the Jordan Curve Theorem, there exists a retraction $%
r:X\rightarrow S$. If $\theta $ is a "rotation" of $S$ without periodic
points, then the composition $\theta r:X\rightarrow X$ has no periodic
points and maps $\overline{B}$ into $B$, contrary to the property of the $wLPPP$%
-basic set $B$.

Put
\[ A'=A\cup\bigcup\{D: \text{$D$ is a bounded component of $\mathbb R^2\setminus A$}\}\]
and
\[V'=V \cup\bigcup\{D: \text{$D$ is a bounded component of $\mathbb R^2\setminus A$}\}.\]

Notice that $\overline{V'}=A'$. By~\eqref{eq1}, $A^{\prime }$ is a
subcontinuum of $X$. Obviously, $A^{\prime }$ is locally connected and it
does not separate $\mathbb{R}^{2}$, so $A^{\prime }\in FPP$ by Proposition~\ref{P1}.
 Moreover, since the diameter of each bounded component $D$ of $\mathbb R^2\setminus A$ does not exceed the diameter of $A$, we have $\diam A'\leq \varepsilon$. So, in the case when $A$ separates $\mathbb{%
R}^{2}$, we include $V^{\prime }$ to $\mathcal{V}_{p}$.

\end{proof}

\bigskip

Recall that a continuum is a \textit{local dendrite} if each of its points
has a closed neighborhood which is a dendrite and that dendrites have $FPP$.

In the proof of the next theorem we use the following well known facts.

\begin{proposition}\cite[Theorem 1, p. 354]{Ku}\label{P3} Each simple closed
curve contained in a one-dimensional metric continuum $X$ is a retract of $X$.
\end{proposition}

\begin{proposition}\cite[Theorems 4 and 5, pp. 303-304]{Ku}\label{P4} If a
locally connected metric continuum is not a local dendrite, then it contains
a sequence of simple closed curves with diameters converging to $0$.
\end{proposition}

\begin{theorem} The following conditions are equivalent for a locally
connected one-dimensional metric continuum $X$.

\begin{enumerate}
\item
$X\in LPPP$,
\item
$X\in wLPPP$,
\item
$X$ is a local dendrite,
\item
$X\in LFPP$,
\item
$X\in wLFPP$.
\end{enumerate}
\end{theorem}

\begin{proof} (2) $\Rightarrow $ (3). Suppose $X$ is not a local dendrite.
Then, by Proposition~\ref{P4}, there is a sequence of simple closed curves $%
S_{n}\subset X$ converging (in the Hausdorff metric) to a singleton $\{p\}$.
If $\mathcal{B}$ is an arbitrary open basis in $X$ and $B\in \mathcal{B}$
contains $p$, then $B$ contains a simple closed curve $S=S_{n}$ for some $%
n\in \mathbb{N}$. By Proposition~\ref{P3}, there is a retraction $r:X\rightarrow S$.
 As before, compose $r$ with a map $\theta :S\rightarrow S$ without
periodic points to get a map $f:X\rightarrow X$ without periodic points such
that $f(\overline{B})\subset B$.

All remaining implications are obvious.
\end{proof}

\bigskip

Next we present an example of a plane continuum $X\in wLPPP \setminus
wLFPP$ having an open basis $\mathcal{B}$ such that, for each map $%
f:X\rightarrow X\,$ and each $B\in \mathcal{B}$, $f$ has a periodic point
of period at most two in $B$. It was observed in~\cite{KOU} that a harmonic
fan $F$ is a plane continuum in $LPPP\setminus wLFPP$ but it is easy
to see that in $F$ it is not possible to find a bound for the periods, in
the described sense.

\begin{example}\label{ex1}
We use the Ma\'{c}kowiak's example of a chainable,
hereditarily decomposable continuum $M$ which is rigid, i.e., it admits no non-constant,
non-identity maps between subcontinua~\cite{Ma}. We locate $M$ in the plane semi-annulus
\[
\{z=(x,y):1\leq \left\vert z\right\vert \leq 2, \, 0\leq y\}
\]
so that there
are symmetric points
\begin{center}
$a\in M\cap ([1,2]\times \{0\})$ and $-a\in M\cap ([-2,-1]\times \{0\})$.
\end{center}

We may also ask that

\begin{center}
$\{a\}=M\cap ([1,2]\times \{0\})$ and $\{-a\}=M\cap ([-2,-1]\times \{0\})$.
\end{center}

In fact, we may assume that $a=(2,0)$.

Define $Y=M\cup (-M)$. Let $%
h:Y\rightarrow Y$ be the homeomorphism given by $h(y)=-y$. Then $h$ has no
fixed points in $Y$ and $h\circ h$ is the identity on $Y$.

Fix an open set $U$ in $Y$ such that $a\in U$ and $V=-U$ is disjoint from  $U$.

Our example $X$ is an infinite wedge of homeomorphic copies $Y_n$ of $Y$, $n\in \mathbb N$, intersecting at the point $v=(0,0)\in \mathbb R^2$.

To describe $X$ more precisely, consider  an auxiliary
sequence of convex triangles $T_{n}$ in $\mathbb{R}^{2}$ such that $v$ is a
vertex of each $T_{n}$, $T_{n}\cap T_{m}=\{v\}$ if $n\neq m$ and $\lim
T_{n}=\{v\}$ in the Hausdorff metric. For each $n\in \mathbb{N}$, let $%
Y_{n}$ be a homeomorphic copy of $Y$ located in $T_n$ such that the point $%
a_{n}$ corresponding  to $a$ equals $v$. Let $-a_{n}$ be the point in $Y_{n}$
that corresponds to $-a$. We may assume that $d(v,-a_{n})=\max \{d(v,x):x\in
Y_{n}\}$, where $d$ is a metric in $Y_n$. Let $M_{n}$ (resp., $-M_{n}$, $U_{n}$ and $V_{n}$) be the
 subcontinuum of $Y_{n}$ that corresponds to $M$
(resp., $-M$, $U$ and $V$) and let $h_{n}:Y_{n}\rightarrow Y_{n}$ be the
homeomorphism that corresponds to $h$. Then
\begin{multline*}
U_{n}\cap V_{n}=\emptyset, \quad -M_{n}=h_{n}(M_{n}), \quad M_{n}=h_{n}(-M_{n}), \quad -a_{n}=h_{n}(a_{n}), \\
V_{n}=h_{n}(U_{n}), \quad\text{$h_{n}\circ h_{n}$ is the identity on $M_{n}$}\\
\text{and}\quad \text{$h_n$ is a fixed-point-free map}.
\end{multline*}

Define
\[X=\bigcup \{Y_{n}:n\in \mathbb{N}\}.\]

In order to see that $X\notin wLFPP$, take any open basis $\mathcal{B}$ of $X$.
Fix $B\in \mathcal{B}$ with $v\in B$. Let $n\in \mathbb{N}$ be such
that $Y_{n}\subset B$. Let $f:X\rightarrow X$ be defined as

\begin{center}
$f(x)=\left\{
\begin{array}{cc}
h_{n}(v)\text{,} & \text{if }x\notin Y_{n}\text{,} \\
h_{n}(x)\text{,} & \text{if }x\in Y_{n}\text{.}%
\end{array}%
\right. $
\end{center}

Clearly, $f$ is a continuous fixed-point-free map such that $f(\overline{B})\subset
B$. Therefore, $X\notin wLFPP$.

\smallskip

We are going to define a neighborhood basis $\mathcal{V}_{p}$ at each point $%
p\in X$ such that $\bigcup\limits_{p\in X}\mathcal{V}_{p}$ is a basis for $%
wLPPP$. In every case $\mathcal{V}_{p}$ is of the form
\[
\mathcal{V}_{p}=\{B(p,\varepsilon )\cap X:0<\varepsilon <\varepsilon _{p}\},
\]
 where $B(p,\varepsilon )$ is the Euclidean $\varepsilon $-ball in $\mathbb{R}^{2}$. So, we need to say how to choose $\varepsilon _{p}$. For $%
p=v$, let $\varepsilon _{p}=1$. If $p\in Y_{n}$ $\backslash $ $%
\{a_{n},-a_{n}\}$ for some $n\in \mathbb{N}$, we choose $\varepsilon _{p}>0$
such that
\[
B(p,2\varepsilon _{p})\cap X\subset Y_{n}\setminus
\{a_{n},-a_{n}\},
\]
\[ B(p,2\varepsilon _{p})\cap  -M_{n}=\emptyset
 \quad\text{for}\quad p\in M_{n} \quad\text{and}\quad B(p,2\varepsilon _{p})\cap M_{n}=\emptyset \quad\text{for}\quad p\in -M_{n}.
\]
 Finally, if $p=-a_{n}$ for some $n$, we choose $\varepsilon _{p}>0$ such
that $B(p,\varepsilon _{p})\cap X\subset V_{n}$.

We will need the following property of $X$.

\begin{property}\label{pro1} Let $f:X\rightarrow X$ be a continuous map. If there exists a
nondegenerate subcontinuum $A$ of $M_{n}$ (resp., $-M_{n}$) such that $%
f(A)=\{q\}$ for some $q\neq v$, then $f(M_{n})=\{q\}$ (resp., $%
f(-M_{n})=\{q\}$).
\end{property}

To prove Property~\ref{pro1}, consider the hyperspace of subcontinua $C(X)$ of $X$ with the Hausdorff
metric and use the fact that given subcontinua $A$ and $B$ of $X$ such that
$A\subset B$, there exists a continuous map $\alpha :[0,1]\rightarrow C(X)$
such that $\alpha (0)=A$, $\alpha (1)=B$ and $\alpha (s)\subset \alpha (t)$
if $0\leq s\leq t\leq 1$~\cite[Theorem 14.6]{IN}.

In our setting, let $\alpha
:[0,1]\rightarrow C(X)$ be such such that $\alpha (0)=A$, $\alpha (1)=M_{n}$ and $%
\alpha (s)\subset \alpha (t)$ if $0\leq s\leq t\leq 1$. Let $r\in \mathbb{N}$
be such that $q\in Y_{r}=M_{r}\cup (-M_{r})$. We may assume that $q\in M_{r}$%
. Let $t_{0}=\max \{t\in \lbrack 0,1]:f(\alpha (t))\subset M_{r}\}$. Notice
that $f|_{\alpha (t_{0})}: M_{n}\supset\alpha (t_{0}) \rightarrow M_{r}$. By
the rigidity of $M$, the map $f|_{\alpha (t_{0})}$ is either an embedding or
a constant map. Since $A\subset \alpha (t_{0})$, $f|_{\alpha (t_{0})}$
cannot be one-to-one, so $f|_{\alpha (t_{0})}$ is a constant map. Thus, $%
f(\alpha (t_{0}))=\{q\}$.

We claim that $t_{0}=1$. Suppose to the contrary that $t_{0}<1$. If $q\in M_{r}\setminus \{-a_{r}\}$ then, since $q\neq v$ and $M_{r}\setminus \{-a_{r},v\}$ is open in $X$,
there exists $t_{0}<t<1$ such
that $f(\alpha (t))\subset M_{r}\setminus\{-a_{r},v\}$,
contradicting the choice of $t_{0}$. Hence, $q=-a_{r}$. Since $q$ belongs to the open subset $Y_{r}\setminus\{v\}$ of $X$,  there
exists $t_{0}<t<1$ such that $f(\alpha (t))\subset Y_{r}\setminus
\{v\}$. Now,
\[
\bigl(M_{r}\cap f(\alpha (t))\bigr)\cap \bigl(-M_{r}\cap f(\alpha
(t))\bigr)=\{-a_{r}\},
\]
so both sets $M_{r}\cap f(\alpha (t))$ and $%
-M_{r}\cap f(\alpha (t))$ are subcontinua of $f(\alpha (t))$ and each of
them is a retract of $f(\alpha (t))$. Let $\varphi :f(\alpha (t))\rightarrow
M_{r}\cap f(\alpha (t))$ be the retraction that sends $-M_{r}\cap f(\alpha
(t))$ to $q$. By the rigidity of $M$, the map $\varphi \circ f:\alpha
(t)\rightarrow M_{r}\cap f(\alpha (t))$, is either an embedding or a
constant map. Since $A\subset \alpha (t)$ and $\varphi (f(A))=\{q\}$, $%
\varphi \circ f$ is the constant map that sends $\alpha (t)$ to $q$. Thus,
for each $x\in \alpha (t)$ that satisfies $f(x)\in M_{r}$, we have $f(x)=q$.
Similarly, $f\bigl(\alpha (t)\cap f^{-1}(-M_{r})\bigr)=\{q\}$. Thus $%
f(\alpha (t))=\{q\}$. This contradicts the choice of $t_{0}$ and shows that
$t_{0}=1$. Hence, $f(M_{n})=f(\alpha (1))=\{q\}$ and the proof of Property~\ref{pro1} is finished.

\

In order to show that $\mathcal{B}=\bigcup\limits_{p\in X}\mathcal{V}_{p}$
is a basis for $wLPPP$, let $f:X\rightarrow X$ be a continuous map, and let $%
p\in X$ and $0<\varepsilon \,<\varepsilon _{p}$ be such that $%
f(\overline{B})\subset B$, where $B=B(p,\varepsilon )\cap X$.

We are going to to show that $f$ has a periodic point of period at most two.
We analyze 3 cases.

\begin{case} $p\in (M_{n}\setminus\{v,-a_{n}\})\cup
(-M_{n}\setminus  \{v,-a_{n}\})$ for some $n\in \mathbb{N}$.
\end{case}

In this case, we suppose that $p\in M_{n}\setminus \{v,-a_{n}\}$, the
other case is similar. By the choice of $\varepsilon _{p}$, $\overline{B}\subset
M_{n}\setminus \{v,-a_{n}\}$. By~\cite[Theorem 14.6]{IN}, there exists a
nondegenerate subcontinuum $A$ of $X$ such that $p\in A\subset B$. By the rigidity of $M_n$, $f|_{A}$ is
either the natural embedding or a constant map. If $f(a)=a$
for each $a\in A$, we
are done. If  $f(a)=q\in M_{n}$ for
each $a\in A$, then $f(M_{n})=\{q\}$ by Property~\ref{pro1}.
Hence, $q$ is a fixed point of $f$ in $B$.

\begin{case}
$p=-a_{n}$ for some $n\in \mathbb{N}$.
\end{case}

By the choice of $\varepsilon _{p}$, $B\subset V_{n}\subset Y_{n}\setminus\{v\}$.

We consider two subcases.

\begin{claim} There exists a nondegenerate subcontinuum $E$ of $B$
such that $f(E)$ is a one-point set.
\end{claim}

Suppose that $\{q\}=f(E)$. Note that $q\in B$.
There exists a nondegenerate subcontinuum $G$ of $E\setminus
\{v,-a_{n}\}$. We may assume that $G\subset M_{n}$. By Property~\ref{pro1}, $f(M_{n})=\{q\}$,  in particular, $f(-a_{n})=q$. If $q\in M_{n}$, then $q$
is a fixed point for $f$ in $B$ and we are done. Suppose then that $q\in
-M_{n}\setminus\{v,-a_{n}\}$. Let $\eta >0$ be such that
\[
B(-a_{n},\eta )\cap B(q,\eta )=\emptyset \quad\text{and}\quad B(q,\eta )\cap X\subset
-M_{n}\setminus\{v,-a_{n}\}.
\]
 Let  $0<\delta
<\min \{\varepsilon ,\eta \}$ and $f(B(-a_{n},\delta )\cap X)\subset
B(q,\eta )$. Since $B(-a_{n},\delta )\cap $ $-M_{n}$ is a nonempty open
subset of the continuum $-M_{n}$, there exists a nondegenerate subcontinuum $%
K$ of $-M_{n}$ such that $K\subset B(-a_{n},\delta )$. Then $f(K)\subset
-M_{n}\setminus\{v,-a_{n}\}$. By the rigidity of $-M_{n}$, $f|_{K}$
is the identity on $K$ or $f(K)$ is a singleton. Since $K\subset
B(-a_{n},\delta )$ and $f(K)\subset B(q,\eta )$, $f|_{K}$  cannot be the
identity on $K$. Thus, $f(K)=\{w\}$, for some $w\in -M_{n}\setminus
\{v,-a_{n}\}$ and,  by Property~\ref{pro1}, $f(-M_{n})=\{w\}=\{q\}$. Hence, $q$ is a fixed point for $f$ in $B$.

\begin{claim}\label{sub2} For every nondegenerate subcontinuum $E$ of $B$, $f(E)$
is nondegenerate.
\end{claim}

Since $B$ is open and nonempty in $X$, there exists a nondegenerate
subcontinuum  $E\subset B$. Since $E\setminus
\{v,-a_{n}\}$ is a nonempty open subset of $E$, we may assume that $E\cap
\{v,-a_{n}\}=\emptyset $. Moreover, the set $f(E)$ being nondegenerate, $E\setminus
f^{-1}(\{v,-a_{n}\})$ is a nonempty open subset of $E$. Thus, we may also
assume that $E\cap f^{-1}(\{v,-a_{n}\})=\emptyset $. Then $E\subset M_{n}$
or $E\subset -M_{n}$ and $f(E)\subset M_{n}$ or $f(E)\subset -M_{n}$. We may
assume that $E\subset M_{n}$. Recall that $h_{n}(-M_{n})=M_{n}$ and $%
h_{n}(M_{n})=-M_{n}$. Thus, $f(E)\subset M_{n}$ or $h_{n}(f(E))\subset M_{n}$%
. Since $f(E)$ is nondegenerate, the rigidity of $M_{n}$ implies that $%
f(e)=e $ for each $e\in E$ or $h_{n}(f(e))=e$ for each $e\in E$. In the
first case, there are fixed points  of $f$
in $B$. Hence, we may assume that  $%
h_{n}(f(E))\subset M_{n}$ and $h_{n}(f(e))=e$ for each $e\in E$. Fix a point
$e_{0}\in E\subset B\subset V_{n}$. Since $h_{n}(f(e_{0}))=e_{0}$, we get $%
f(e_{0})=h_{n}(e_{0})\in h_{n}(V_{n})=U_{n}$. But, $f(e_{0})\in f(B)\subset
B\subset V_{n}$. Thus, $f(e_{0})\in U_{n}\cap V_{n}$, a contradiction. This
completes Subcase~\ref{sub2}.

\begin{case}
$p=v$.
\end{case}

Here, we may assume that $f(p)\neq p$ and consider 2 subcases.

\begin{claim} $f(p)\in Y_{n}\setminus\{v,-a_{n}\}$ for some $%
n\in \mathbb{N}$.
\end{claim}

In this subcase we will see that $f$ is a constant map. We may assume that $%
f(p)\in M_{n}$, the case  $f(p)\in -M_{n}$ being similar. Let $m\in \mathbb{%
N}$. Since $f^{-1}(M_{n}\setminus\{v,-a_{n}\})$ is an open subset of
$X$ containing $p$, there exists a nondegenerate subcontinuum $A$ of $M_{m}$
such that $p\in A\subset f^{-1}(M_{n}\setminus\{v,-a_{n}\})$. Then
$$
f(A)\subset M_{n}\setminus\{v,-a_{n}\}\quad\text{and}\quad f|_{A}:
M_{m}\supset A\rightarrow M_{n}$$
 By the rigidity of $M$, $f|_{A}$ is either a
constant map or the natural embedding. Since $f(a_{m})=f(p)\neq v=a_{n}$, $%
f|_{A}$ is not the natural embedding. Hence, $f(a)=f(p)$ for each $a\in A$.
By Property~\ref{pro1}, $f(M_{m})=\{f(p)\}$. By an analogous argument, we see that $%
f(-M_{m})=\{f(p)\}$. Therefore, $f(Y_{m})=\{f(p)\}$ for each $m\in \mathbb{N}
$, so $f$ is constant and $f(p)$ is a fixed point of $f$
in $B$.

\begin{claim}
$f(p)=-a_{n}$ for some $n\in \mathbb{N}$.
\end{claim}

In this case $-a_{n}\in B$ and since $d(v,-a_{n})=\max \{d(v,x):x\in Y_{n}\}$, we have $Y_{n}\subset B$.
Let $\alpha :[0,1]\rightarrow C(M_{n})$ be a continuous
map such that $\alpha (0)=\{p\}$, $\alpha (1)=M_{n}$ and $\alpha
(s)\subsetneq \alpha (t)$ if $0\leq s<t\leq 1$. If $\alpha (t)$ contains a nondegenerate
subcontinuum $A$ such that $p\in A$ and $f(A)$ is a one-point set, then $f(M_{n})=\{-a_{n}\}$ by Property~\ref{pro1}. So, $f(-a_{n})=-a_{n}
$ is a fixed point for $f$ in $B$. Suppose then that for
each $t>0$ and, for each nondegenerate subcontinuum $A$ of $\alpha (t)$ containing $p$,
 $f(A)$ is nondegenerate. Let $t_{0}=\max \{t\in \lbrack 0,1]:f(\alpha
(t))\subset -M_{n}\}$. Then $f(\alpha (t_{0}))\subset -M_{n}$.

We claim that
\begin{equation}\label{eq2}
f(x)=h_{n}(x) \quad\text{for each}\quad x\in \alpha (t_{0}).
\end{equation}
If $t_{0}=0$, then $\alpha (t_{0})=\{p\}$ and $%
f(p)=-a_{n}=h_{n}(a_{n})=h_{n}(p)$. Now, if $t_{0}>0$, $\alpha (t_{0})$ is a
nondegenerate subcontinuum of $M_{n}$ such that $f|_{\alpha (t_{0})}: M_{n}\supset \alpha
(t_{0})\rightarrow -M_{n}$, so $h_{n}\circ f|_{\alpha
(t_{0})}:M_{n}\supset \alpha
(t_{0})\rightarrow M_{n}$. The rigidity of $%
M_{n}$ implies that $h_{n}\circ f|_{\alpha (t_{0})}$ is either a constant
map or the identity on $\alpha (t_{0})$. Since $t_{0}>0$, $f(\alpha (t_{0}))$
is nondegenerate, so $h_{n}(f(\alpha (t_{0})))$ is nondegenerate. Thus, $%
h_{n}\circ f|_{\alpha (t_{0})}$ is the identity on $\alpha (t_{0})$. Hence,
for each $x\in \alpha (t_{0})$, $h_{n}\circ f(x)=x$ and $f(x)=h_{n}(x)$.

Now, we will see that $p\in f(\alpha (t_{0}))$. Otherwise,
 \[
 -a_{n}=h(a_{n})=h_{n}(p)\notin
h_{n}\bigl(f(\alpha (t_{0}))\bigr)=\alpha (t_{0})\quad\text{and}\quad t_{0}<1.
\]
Thus,
\[\alpha
(t_{0})\cap f(\alpha (t_{0}))\subset (M_{n}\cap -M_{n})\setminus
\{p,-a_{n}\}=\emptyset.
\]
Let $W_{1}$ and $W_{2}$ be disjoint open subsets
of $X$ such that $\alpha (t_{0})\subset W_{1}$ and $f(\alpha (t_{0}))\subset
W_{2}\subset Y_{n}\setminus\{p\}$. Then there exists $t_{0}<t_{1}<1$
such that $\alpha (t_{1})\subset W_{1}$ and $f(\alpha (t_{1}))\subset W_{2}$%
. By the choice of $t_{0}$, there exists a point $y_{0}\in f(\alpha (t_{1}))\setminus -M_{n}\subset M_{n}\setminus -M_{n}$. Let $x_{0}\in
\alpha (t_{1})$ be such that $y_{0}=f(x_{0})$. Since $\alpha (t_{1})\setminus f^{-1}(-M_{n})$ is an open subset of the nondegenerate
continuum $\alpha (t_{1})$ and it contains $x_{0}$, there exists a
nondegenerate subcontinuum $A$ of $\alpha (t_{1})$ such that $x_{0}\in
A\subset \alpha (t_{1})\setminus f^{-1}(-M_{n})$. Then
\[
A\subset
\alpha (t_{1}), \quad f(A)\subset W_{2}\subset Y_{n}\setminus \{p\} \quad\text{and}\quad
f(A)\cap -M_{n}=\emptyset.
\]
 Thus, $f(A)\subset M_{n}$ and $f|_{A}:
M_{n}\supset A\rightarrow M_{n}$. By the rigidity of $M_{n}$, $f(a)=a$ for each $a\in
A$ or $f(A)$ is a singleton. In the latter case, by Property~\ref{pro1}, $f(M_{n})=\{-a_{n}\}$
and $-a_{n}$ is a fixed point for $f$ in $B$. Suppose now that $f(a)=a$
for each $a\in A$. In particular,
\[
x_{0}=f(x_{0})\in \alpha (t_{1})\cap
f(\alpha (t_{1}))\subset W_{1}\cap W_{2},
\]
 a contradiction. We have shown
that $p\in f(\alpha (t_{0}))$.

Let $x_{1}\in \alpha (t_{0})$ be such that $%
p=f(x_{1})=h_{n}(x_{1})$ (by~\eqref{eq2}). Then
$$
x_{1}=h_{n}(p)=-a_{n}\quad\text{and}\quad f(f(p))=f(-a_{n})=f(x_{1})=p.$$
 Therefore, $p$ is
a periodic point for $f$ in $B$ of order 2.

This completes the proof that for each continuous map $f:X\rightarrow X$ and
for each $B\in \mathcal{B}$ such that $f(\overline{B})\subset B$, $f$ has a
periodic point in $B$ of order at most 2.

In particular, we have shown that $X\in wLPPP\setminus wLFPP$.
\flushright $\square$
\end{example}

\bigskip

It is easy to check that  continuum $Y$ employed in Example~\ref{ex1} has not the fixed-point property but
any continuous self-map of $Y$ has a point of period at most 2. The example shows that  the infinite wedge of copies of $Y$ localize the properties.
Looking for a respective  locally connected example, we can recall that all connected compact polyhedra $X$ with trivial
odd-dimensional homology groups $H_{2i+1}(X,\mathbb{Q})$ have the periodic
point property [2, Proposition 4.4, p. 232]. Specifically, for any
fixed-point-free self-map of a $2$-sphere there is a point of period $2$.
Hence, we pose the following question.

\begin{question}
Is there a locally connected continuum $X\in LPPP\setminus LFPP$ ($X\in wLPPP\setminus wLFPP$)? Is the infinite wedge of $2$-spheres such  continuum?
\end{question}

\begin{question}
Are properties $LPPP$ and $LFPP$ ($wLPPP$ and $wLFPP$) equivalent for ANR%
-continua?
\end{question}

\begin{question}
Are properties $LPPP$ and $wLPPP$ ($LFPP$  and $wLFPP$) equivalent for ANR-continua?
\end{question}

\bigskip

\begin{example}\label{ex2}
\emph{There exists a continuum}
$$X \in wLFPP \setminus LPPP \subset wLFPP \setminus LFPP.$$
We consider the Cook's continuum $X$ constructed in~\cite[Theorem 8]{Co}, which is one-dimensional, nonplanar, hereditarily indecomposable and it is also rigid (it admits no non-constant, non-identity maps between subcontinua).

In order to see that $X \in wLFPP$, take any open basis $\mathcal B$ of nonempty sets.
Let $B \in \mathcal B$ and $f : X \to  X$ be a continuous map such that $f(\overline B) \subset B$.
By the rigidity of $X$, $f$ is either the identity map or it is a constant map. In the
first case, any point in $B$ is a fixed point for $f$. In the second case, there exists
$q\in X$ such that $f(x) = q$ for each $x \in X$. Notice that $q \in B$, so $q$ is a fixed
point for $f$ in $B$.

Now we check that $X \notin LPPP$. Take any open basis $\mathcal B$ for $X$. By~\cite{Maz}, there exists a Cantor set $C_0\subset X$ which contains at most one point of each composant of $X$. Fix a point $p_0 \in C_0$ and let $B \in\mathcal B$ be such that $p_0 \in B$
and $\overline  B \neq  X$. Since $B$ is open, there exists a Cantor set $C \subset B \cap C_0$. Then $C$
has also the property that it contains at most one point of each composant of $X$. Observe that there is a retraction $r :\overline B \to C$. Indeed,
if $D$ is the quotient space of the decomposition of $\overline B$ into components and $q:\overline B \to D$ is the quotient map, then $q|C: C\to q(C)$ is a homeomorphism and $q(C)$ is a non-empty closed subset of a zero-dimensional space $D$, so there is a retraction $f:D\to q(C)$~\cite[Problem 1.3.C, p. 22]{Eng}. Then put $r=q^{-1}\circ f\circ q$. Let  $\sigma:C\to C$ be a periodic-point-free map (any minimal map of $C$, say the adding machine on $C=\{0, 1\}^\infty$, can be used).  Then the composition $\sigma\circ r:\overline B \to \overline B$ is periodic-point-free.
\flushright $\square$
\end{example}

\begin{example}\label{ex3}
\emph{There exists a (non-planar) continuum}
$$Z \in wLPPP \setminus (LPPP\cup wLFPP).$$

Consider the  Cook's continuum $X$, that we used in Example~\ref{ex2}. Fix two different points $p, q \in X$.
Let $X'$ be a disjoint topological  copy of $X$, with corresponding points $p', q'$, and let $Y$  be a continuum obtained from $X\cup X'$ by identifying $p$ with $q'$ and $q$ with $p'$. Take an infinite  bouquet
$$Z=\bigcup\{Y_n: n\in\mathbb N\}$$
of copies $Y_n$ of $Y$, joined at the point $v=p_n$ for each $n$, where  $p_n\in Y_n$ corresponds to $p$.

We can mimic the argument from Example~\ref{ex1} to obtain
 $Z \in wLPPP\setminus wLFPP$.

In order to show that $Z \notin LPPP$, take an open basis $\mathcal B$ of $Z$. Let $B \in \mathcal B$ be such that
$$\emptyset\neq \overline B\subset Y_1\setminus \{v\}.$$
Now, the proof from Example~\ref{ex2} that $\overline B\notin PPP$ can be repeated  step-by-step.
\flushright $\square$
\end{example}

\begin{acknowledgement}
The results of this paper originated during the
7th Workshop on Continuum Theory and Hyperspaces in Quer\'{e}taro, M\'{e}%
xico, July 2013. The authors wish to thank the participants for useful
discussions, particularly, Roc\'{\i}o Leonel, Jorge M. Mart\'{\i}nez, Norberto
Ordo\~{n}ez and Jos\'{e} A. Rodr\'{\i}guez. This paper was partially
supported by the projects "Hiperespacios topol\'{o}gicos (0128584)" of
Consejo Nacional de Ciencia y Tecnolog\'{\i}a (CONACYT), 2009 and "Teor\'{\i}%
a de Continuos, Hiperespacios y Sistemas Din\'{a}micos" (IN104613) of
PAPIIT, DGAPA, UNAM.
\end{acknowledgement}

\bibliographystyle{amsplain}

\end{document}